\documentclass[14pt]{article}
\usepackage[cp1251]{inputenc}
\usepackage[english]{babel}
\usepackage{amsfonts,amssymb,amsmath,amstext,amsbsy,amsopn,amscd,amsthm,graphicx,euscript}
\usepackage{graphics}
\newtheorem{thm}{Theorem}
\newtheorem{lm}{Lemma}

\newcounter{tdfn}
\newcounter{trk}
{\endtrivlist}

\def\:{\colon}

\def\0{{\mathbf 0}}
\def\1{{\mathbf 1}}

\title{An elementary proof that classical braids embed into virtual braids}

\author{V.O.Manturov \footnote{The present work is supported by the Laboratory of Quantum Topology ot the Chelyabinsk State University (Russian Govenrment Grant No,
 14.Z50.31.0020) and also by RFBR grants 13-01-00830,14-01-91161,
14-01-31288.}}

\date{}

\begin{document}

\maketitle

AMS MSC 57025, 20F36, 57027

{\Large

\begin{abstract}
The aim of the present note is to show that the natural map from classical braids to virtual braids is an inclusion; this proof does not use any complete invariants of classical braids; it is based on the projection from virtual braids to classical braids (similar to the one given in \cite{Projection}); this projection is the identity map on the set of classical braids. The projection is well defined not only for the virtual braid group but also for the quotient group of the virtual (pure) braid group by the so-called virtualization move. The idea of this projection is closely related to the notion of parity \cite{Parity} and to the groups $G_{n}^{k}$, introduced by the author \cite{Great}.
\end{abstract}

Virtual braids are a natural generalization of classical braids. Classical braids form a subgroup of the virtual braid group: if two classical braids are equal in the virtual braid group, then they are equal as classical braids. This fact was first given in \cite{FRR}. In the present paper, we give an elementary proof of a stronger statement. Namely, we prove that the composition of  the natural map from the classical (pure) braid group to virtual (braid) group with a projection to some to some natural quotient group of the virtual braid group is an inclusion.

Without loss of generality, we shall work with pure (classical and virtual) braids only. Moreover, later we shall restrict ourselves to some finite-index subgroup og the pure virtual braid group.

Let us fix a positive integer  $n$. By a {\em set of signs} we mean an {\em ordered collection of signs}: for each pair $i,j$ of distinct integers from $\{1,\dots,n\}$ the number $s_{ij}\in \{\pm 1\}$ is defined, so that $s_{ij}=-s_{ji}$.

We say that a set of signs is {\em realizable} if there exists such a set of distinct real numbers $N_{j},j\in \{1,\dots, n\}$ for which $s_{ij}=sign (N_{j}-N_{i})$.


We say that for some set $S$ of signs, two different indices $i,j\in \{1,\dots, n\}$ are  {\em adjacent} if for each  $k$ we have $s_{ik}=s_{jk}$.

We define the pure virtual braid group  $PB_{n}$ according to \cite{Bardakov} as the group having the following presentation:

The generators are
$$a_{ij},1\le i\neq j\le n.$$

The relations are:
$$a_{i,j}a_{i,k}a_{j,k}=a_{j,k}a_{i,k}a_{i,j}, i,j,k\;\;\;\mbox{distinct}; $$
$$a_{i,j}a_{k,l}=a_{k,l}a_{i,j}, i,j,k,l \;\;\;\mbox{distinct}; $$

Later on, we shall call words in $a_{ij}$ {\em braid-words}. For letters in braid-words, we say that the {\em writhe number} is $+1$ for all the generators: $w(a_{i,j})=+1$, and we set $w(a_{i,j}^{-1})=-1$.

This group is isomorphic to the pure virtual braid group with its standard presentation (see \cite{Bardakov}).

We introduce the following important quotient group of the pure braid group
$${\tilde PB}_{n}=PB_{n}\slash \langle a_{ij}=a_{ji}\rangle,1\le i\neq j\le n;$$
we also define the group $G_{n}^{2}$ to be the quotient group of $PB_{n}$ by the relations $a_{ij}^{2}=1$,
see \cite{Great}.

Having a braid-word $\beta$, we can represent it by a {\em braid diagram} as follows. Let $\beta$ consist of
letters $\beta_{1}\cdots \beta_{k}$. The {\em braid diagram} will consist of $n$ strands each connecting the upper point $(l,1)$ to the lower point $(l,0)$ for $l=1,\dots, n$;

The natural map $o$ from pure classical braid group to the group ${\tilde PB}_{n}$ is defined as follows. Given a pure classical $n$-strand braid  $b$; let us enumerate the strands of $b$ by natural numbers from  $1$ to $n$ according to their endpoints.

Note that unlike Artin's {\em local} enumeration of strands, this enumeration is {\em global}.

Let us start reading the classical crossings of $b$ from the top to the bottom.
 Assume the undercrossing strand approaches this crossing from the top-right; then we denote the undercrossing line by $j$, denote the overcrossing line by $i$, and associate with this crossing the generator $a_{ij}$. If the undercrossing strand approaches this crossing from the top-left, then we denote the undercrossing line by $j$, denote the ovecrossing line by $i$, and associate with this crossing the generator $a_{ij}^{-1}$, see Fig. \ref{bardakovscrossings}.

 Note that the above description can be also thought of as the description of the presentation
of  {\em pure virtual braids} by its planar diagrams: we ignore virtual crossings.

 \begin{figure}
 \centering\includegraphics[width=150pt]{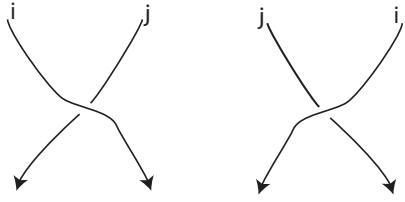}
 \caption{The crossing corresponding to $a_{ij}$ (left); the crossing corresponding to $a_{ij}^{-1}$ (right)}
 \label{bardakovscrossings}
 \end{figure}

For each such classical crossing  we write down the generator $a_{ij}$. Note that no generators correspond to virtual crossings; we have written a word in $a_{ij}$ and $a_{ij}^{-1}$; let us denote the resulting word by $o(b)$; for more details, see \cite{Bardakov}.

We say that two braid-words are {\em virtualization-equivalent} if they are obtained from each other by a sequence of
{\em virtualizations}, i.e., changes $a_{ij}\to a_{ji}$.

The geometrical interpretation of the virtualization move is shown in Fig. \ref{vir}.
\begin{figure}
\centering\includegraphics[width=200pt]{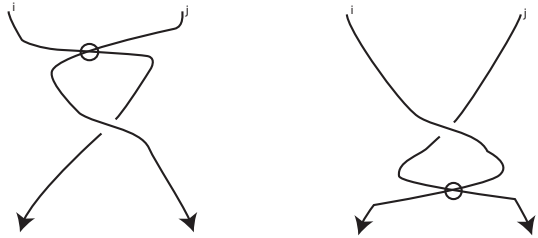}
\caption{The virtualization move}
\label{vir}
\end{figure}

Note that $o(b)$ depends only on the combinatorial structure of classical crossings only. We say that two virtual braids $b,b'$ are {\em detour-equivalent} (or {\em equal}) if the two words $o(b)$ and $o(b')$ coincide. Later on, we shall not make any difference between detour-equivalent virtual braids; in particular, we shall say that $b$ is {\em classical} if there exists a classical braid $b'$ which is detour-equivalent to $b$.

Note that when we apply a virtualization to a crossing, we do not change the writhe number of the crossing but the overpass and underpass strands change their roles.

By definition, virtualization-equivalent braid-words represent the same element in the group ${\tilde P}B_{n}$.

The aim of the present paper is to prove the following
\begin{thm}
The map $o$ defines an inclusion of the classical $n$-strand braid group to the group  ${\tilde PB}_{n}$.
\end{thm}

Let us first define the right {\em action} of the group $G_{n}^{2}$ on $S$ (and hence, the actions of  $PB_{n}$ and ${\tilde PB}_{n}$ on $S$) according to the evident rule: the generator $a_{ij}$ switches the signs of $s_{ij}$ and $s_{ji}$, and does not change the other signs.

Let us fix the set $B$ of signs corresponding to the set of numbers  $1,2,\dots, n$, in such a way that $s_{ij}=+1$ if and only if $j>i$.

%

Assume an element of the group $PB_{n}$ is given as a word $\beta=\beta_{1}\beta_{2}\dots \beta_{l}$
(here each $\beta_{l}$ is some $a_{ij}$ or $a_{ij}^{-1}$). This word will act on the set of signs $B$ on the right. More precisely, we shall consider the sets of signs
$B, \beta_{1}B,\beta_{2}\beta_{1}B,\beta_{3}\beta_{2}\beta_{1}B\dots$.
Each classical generator deals with exactly two indices. We say that a classical generator $\beta_{j}$ in the word $\beta$ is {\em good}, if its indices are {\em adjacent} in the word ${\bar \beta}B$  which precedes the action of this generator. It is clear that these two strands will remain adjacent in the word $\beta_{j}{\bar \beta} B$ that we obtain after the action of $\beta_{j}$.

This action is illustrated in Fig. \ref{action}. In the left part, we show the action of a classical (realizable) braid,
and in the right hand side we show the action of a virtual braid; in the left part, we show the realization at each step; in the right part, $s(1,2)=+1,s(1,3)=-1,s(2,3)=+1$ is not realizable.

The main idea of the proof of the main theorem is that the property of being ``good'' is closely related to the ``weak parity'' property widely used in virtual knot theory \cite{IMN}. Here classical braids are ``even'' and if two ``even'' ones are equivalent then they are equivalent within even ones (for more about parity, see \cite{Parity}).

\begin{figure}
\centering\includegraphics[width=350pt]{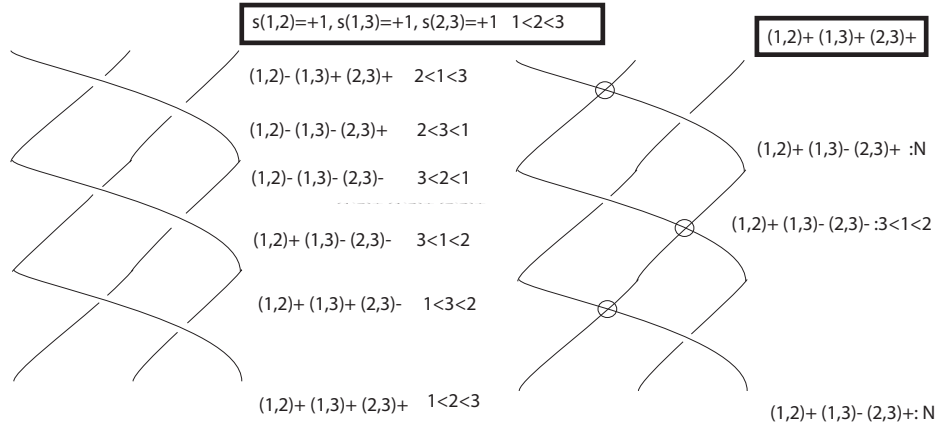}
\caption{A braid acting on a set of signs}
\label{action}
\end{figure}

More precisely, the following lemmas take place:

\begin{lm}
Let $\beta$ be a braid word, and assume it contains two neighbouring letters $a_{ij}$ and $a_{i,j}^{-1}$.
Then the two corresponding crossings are either both good or both bad.\label{lem2}
\end{lm}

\begin{lm}
If  a braid-word $\beta$ admits a third Reidemeister move so that
the three adjacent letters $\beta_{m}\beta_{m+1}\beta_{m+2}$ are transformed into
$\beta'_{m+2}\beta'_{m+1}\beta'_{m}$ then for each $j=m,m+1,m+2$ either both crossings
$\beta_{j},\beta'_{j}$ are good or both crossings are bad.

Moreover, among the three crossings $\beta_{m},\beta_{m+1},\beta_{m+2}$ there number of good crossings is $0,1,$ or $3$.\label{lem3}
\end{lm}

\begin{proof}
Let $\beta={\bar \beta}\beta_{m}\beta_{m+1}\beta_{m+2}{\tilde \beta}$.
Assume the strand numbers for $\beta_{m},\beta_{m+1},\beta_{m+2}$ are $(i,j),(i,k),(j,k)$ respectively.
Assume $\beta_{m}$ is good; this means that the strands $(i,j)$ are adjacent after the action of ${\bar \beta}$
on the standard set of signs $B$. When we consider the crossing $\beta'_{n}$ in $\beta'={\bar \beta}\beta'_{m+2}\beta'_{m+1}\beta'_{m}$, the strands corresponding tho this crossing are adjacent as well because the action of $\beta'_{m+2},\beta'_{m+1}$ changes the signs of $(j,k)$ and $(i,k)$ and does not change other adjacencies.

The cases of pairs $\beta_{m+1},\beta'_{m+1}$ and $\beta_{m+2},\beta'_{m+2}$ are considered analogously.

Now, we have to show that if two of the three crossings $\beta_{m},\beta_{m+1},\beta_{m+2}$ are good, then so is the third one.
We consider only one case. The other cases are considered similarly.

Denote by $S_{p}$ the collection of signs after the action of ${\bar \beta}$.
We assume that the crossings $\beta_{m}$ and $\beta_{m+1}$ are both good. Denote by $S'_{p}$ the collection of signs after the action of ${\bar \beta}{\beta_{m}}$. Then the indices $i,j\in \{1,\dots, n\}$ are adjacent in $S_{p}$ and the indices $i,k$ are adjacent in $S'_{p}$.
This means that in $S'_{p}$ the index $i$ is adjacent to both $j$ and $k$. In particular, this means that for each $l\neq i,j,k$: $s'_{p}(j,l)=s'_{p}(i,l)=s'_{p}(k,l)$. Thus, after the action of ${\beta_{m+1}}$, $j$ and $k$ become adjacent, which means that the crossing $\beta_{m+2}$ is good.

\end{proof}

From the definition, we get the following
\begin{lm}
If a good crossing of some diagram does not take part in some Reidemeister move, then it remains good after this move is performed.\label{lem4}
\end{lm}

Let us define the word $d(\beta)$ as the word obtained from $\beta$ by deleting bad letters. According to the above, we get the following

\begin{lm}
If $\beta,\beta'$ are equal as elements of ${PB}_{n}$ then the words $d(\beta),d(\beta')$ are equal as elements of ${PB}_{n}$.

If $\beta,\beta'$ are equal as elements of ${\tilde PB}_{n}$, then $d(\beta),d(\beta')$
are equal as elements of  ${\tilde PB}_{n}$.
\end{lm}

\begin{proof}
Indeed, one has to check that the lemma holds for the case when $\beta'$ differs from $\beta$ by a braid group relation. For the second Reidemeister move, we use Lemmas \ref{lem2}, \ref{lem4}. For the third Reidemeister move, we use Lemmas \ref{lem3},\ref{lem4}.
Finally, if $\beta'$ differs from $\beta$ by a far commutativity relation $\beta=a_{ij}a_{kl}\to a_{kl}a_{ij}=\beta'$ for all $i,j,k,l$ distinct, then it suffices to see that $d(\beta)$ either coincides with $d(\beta')$ or differs from $d(\beta)$ by the same relation $a_{ij}a_{kl}\to a_{kl}a_{ij}$.
\end{proof}

From the direct check we get the following
\begin{lm}
Let $\beta$ be a braid-word representing a pure virtual braid and assume $\beta$ acts trivially on the set of signs. If all letters of $\beta$ are good then
$\beta$ is virtualization-equivalent to a classical pure braid.
\end{lm}

\begin{proof}
The initial collection of indices $B$ is realizable by definition.
Let $\beta=\beta_{1}\dots \beta_{m}$. By construction, all collections $\beta_{1}B,\beta_{2}\beta_{1}B,\dots,\beta_{m}\cdots\beta_{1}B$ are realizable.

This means that we can construct the braid-word $\beta'$ which is virtualization-equivalent
to $\beta$ step-by step. At each step, we will have no virtual crossings; thus, the resulting braid will be classical.
Having constructed a braid which is virtualization-equivalent to $\beta_{1}\cdots \beta_{j}$,
we have to add a crossing corresponding to $\beta_{j}$. We know which strands it deals with; besides, we know the writhe number of this crossing. So, at each step we get a classical pure braid $\beta'$
which is virtualization-equivalent to $\beta_{1}\beta_{2}\cdots\beta_{k}$. Since $\beta$ (hence, $\beta'$ and all intermediate braids) acts trivially on the set of signs, all lower ends of the resulting braid are in their natural order: $1<2<\cdots<n$.

Finally, we get a braid which is virtualizaiton-equivalent to $\beta$.
\end{proof}

Summing up the above argument and iterating the map $d$, we obtain an action $d^{stab}$ which maps braids from ${\tilde PB}_{n}$ to pure braids which are virtualization-equivalent to classical ones.

Note that the map $d$ is not an idempotent on the set of diagrams of pure virtual braids. Even in the case of the  braid-word $\beta$ with six crossings $a_{13}a_{24}a_{14}a_{14}^{-1}a_{24}^{-1}a_{13}^{-1}$ representing the trivial braid, we see that the letters $a_{14}$ and $a_{14}^{-1}$ are good, so, $d(\beta)=a_{14}a_{14}^{-1},$ and $d^{2}(\beta)$ is the empty word.

Now, assume $\beta$ acts trivially on the set of signs $B$ and let $\beta=\beta_{1}\to \cdots \beta_{k}=\beta'$ be a sequence of braid-words from ${\tilde PB}_{n}$, where each two adjacent
braid-words are related by a relation from ${\tilde PB}_{n}$.
Consider the  sequence of braid-words $\beta=d^{stab}(\beta)=d^{stab}(\beta_{1})\to \cdots \to d^{stab}(\beta_{k})=d^{stab}(\beta')$.
Each two adjacent braid-words $d^{stab}(\beta_{l})$ and $d^{stab}(\beta_{l+1})$ are either identical or they are related by a relation from  ${\tilde PB}_{n}$. Let ${\tilde \beta_{j}}$ be the classical braid-word which is virtualization-equivalent to $d^{stab}(\beta_{j})$ for $j=1,\cdots, k$.
We see that ${\tilde \beta}_{l}$ and ${\tilde \beta}_{l+1}$ are either identical or related by the same relation as $d^{stab}(\beta_{l}),d^{stab}(\beta_{l+1})$. Since all ${\tilde \beta}_{j}$ are classical, we get a sequence of classical Reidemeister moves from $\beta$ to $\beta'$.

This completes the proof of the main theorem.

I am very grateful to A. I.Gaudreau and I.M.Nikonov,  D.P.Ilyutko, and V.G.Bardakov for extremely useful comments improving the quality of
the text.

}

 \end{document}